\documentclass[11pt, a4paper]{amsart}
\usepackage{amsmath,amsthm,amscd,amssymb,latexsym,upref, mathabx, graphicx}
\usepackage[utf8]{inputenc}

\newcommand{\re}{\text{\rm Re\,}}

\newcommand{\bd}{{\mathbb{D}}}
\newcommand{\bn}{{\mathbb{N}}}

\newcommand{\bp}{{\mathbb{P}}}
\newcommand{\bz}{{\mathbb{Z}}}
\newcommand{\bc}{{\mathbb{C}}}

\newcommand{\bt}{{\mathbb{T}}}

\newcommand{\cb}{\mathcal{B}}

\renewcommand{\a}{\alpha}

\renewcommand{\l}{\lambda}

\newcommand{\s}{\sigma}

\newcommand{\ep}{\varepsilon}

\newcommand{\p}{\varphi}

\renewcommand{\d}{\delta}

\newcommand{\oo}{\Omega}

\renewcommand{\gg}{\Gamma}

\newcommand{\ovl}{\overline}
\newcommand{\pt}{\partial}

\newcommand{\lp}{\left(}
\newcommand{\rp}{\right)}

\newcommand{\wih}{\widehat}
\newcommand{\wit}{\widecheck}

\DeclareMathOperator{\di}{\mathrm{dist}}

\allowdisplaybreaks
\numberwithin{equation}{section}

\newtheorem{theorem}{Theorem}[section]

\theoremstyle{definition}
\newtheorem{definition}[theorem]{Definition}
\newtheorem{remark}[theorem]{Remark}

\begin{document}

\title[On the resolvent growth of a Toeplitz operator]
{On local non-tangential growth of the resolvent of a banded Toeplitz operator}

\author{L. Golinskii  }
\address{Mathematics Division, Institute for Low Temperature Physics and
Engineering, 47 Science’s ave., Kharkiv 61103, Ukraine}
\email{golinskii@ilt.kharkov.ua, leonid.golinskii@gmail.com}

\author{S. Kupin}
\address{Institut de Mathématiques de Bordeaux UMR5251, CNRS, Université de Bordeaux, 351 ave. de la Libération, 33405 Talence Cedex, France}
\email{skupin@math.u-bordeaux.fr}

\vspace{1cm}
\begin{abstract}

We study the growth of the resolvent of a Hardy--Toeplitz operator $T_b$ with a Laurent polynomial symbol (\emph{i.e., } 
the matrix $T_b$ is banded), at the neighborhood of a point $w_0\in\pt(\s(T_b))$ on the boundary of its spectrum. We show that 
such growth is inverse linear in some non-tangential domains at the vertex $w_0$, provided that $w_0$ does not belong to a 
certain finite set on the complex plane. 
\end{abstract}


\maketitle

\section*{Introduction}\label{s0}

In paper \cite{gokuvi},  the authors study the growth of resolvent within certain classes of Toeplitz operators $T_b$. 
Recall that, given a function $b\in L^\infty(\bt)$, the functional realization
of the Toeplitz operator $T_b$ on $H^2(\bt)$ is
$$  (T_b f)(t)=P_+\bigl(b(t)f(t)\bigr), \qquad f\in H^2(\bt), 
$$
where $P_+$ is the orthogonal projection from $L^2(\bt)$ onto $H^2(\bt)$, $b$ is a symbol of $T_b$. The matrix realization of the operator is
$$
T_b=\|b_{i-j}\|_{i,j\ge1}: \ell^2(\bn)\to \ell^2(\bn),
$$ 
or, in words, this is a semi-infinite matrix with constant diagonals. Here $b_n$ are the Fourier coefficients of the symbol $b$, \emph{i.e.,}
$$
b(t)=\sum_{n\in\bz} b_n\, t^n, 
$$
where $t\in \bt$.

The problem under consideration is as follows. Given a class $\mathcal{B}=\{b\}$ of symbols, the goal is to find upper bounds of the form
$$ \|(T_b-w)^{-1}\|\le C_b\,\p(\di(w,\s(T_b))), $$
where the function $\p$ depends on $\cb$ (but not on an individual symbol), and the bound holds {\it for all} $w\in\rho(T_b)$, the resolvent set of $T_b$.
A positive constant $C_b$ may depend on $b$.
The main results of \cite{gokuvi} concern such bounds for certain classes $\cb$ and
$$ \p_1(x)=\frac1x, \qquad \p_2(x)=\frac1x\Bigl(1+\frac1x\Bigr). $$
The bound with $\p=\p_1$ is referred to as a Linear Rezolvent Growth (LRG), and one with $\p=\p_2$ as a Quadratic Rezolvent Growth (QRG).
Primarily, the authors deal with banded Toeplitz operators $T_b$ with the Laurent polynomial symbols
\begin{equation}\label{laur}
b(z):=\frac{b_{-m}}{z^m}+\ldots+b_0+\ldots+b_k z^k, \quad m,k\in\bn, \quad b_{-m}b_k\not=0.
\end{equation}

In this note,  we are interested in the local versions of the above result, \emph{i.e.,}  in the upper bounds for 
$\bigl\|(T_b-w)^{-1}\bigr\|$ in a {\it neighborhood of a point} 
$w_0=b(t_0)$, $t_0\in\bt$, 
$$ \oo=\oo(b,w_0,\ep):=B(w_0,\ep)\cap\rho(T_b), \qquad B(w_0,\ep)=\{w:|w-w_0|<\ep\}, $$
for small $\ep>0$.

It is well known \cite[Problem 78]{ha}, that for any bounded linear operator $T$ on a Hilbert space, the boundary of its spectrum is a part of the
essential spectrum
$$ \pt\s(T)\subset \s_{ess}(T). $$
For Toeplitz operators the inclusion looks as $\pt\s(T_b)\subset b(\bt)$. Recall that $\s_{ess}(T_b)=b(\bt)$ is the image of $\bt$ under mapping $b$,
and the above inclusion can be proper. Such examples can be found, \emph{e.g., } in \cite[Figure 7.1, 
p. 166]{bogr}. 
In these cases some parts of $b(\bt)$ lie inside the spectrum of $T_b$, so the local problem makes no sense for such points of $b(\bt)$. 
In what follows, we actually always assume that $w_0\in\pt\s(T_b)$, and so 
$\oo\not=\emptyset$ for each $\ep>0$. 

\begin{definition}
We say that $T_b$ admits the local linear resolvent growth at the point $w_0\in\pt\s(T_b)$, $T_b\in LLRG(w_0)$, if for all small enough
$\ep>0$
\begin{equation}\label{llrg}
\bigl\|(T_b-w)^{-1}\bigr\|\le\frac{C_b}{\di(w,\s(T_b))}\,, \qquad w\in\oo=B(w_0,\ep)\cap\rho(T_b).
\end{equation}
\end{definition}

In general, the local problems (say, in approximation theory) assume certain ``local'' (in a sense) conditions on the symbol, that is, the conditions based on
the behavior of $b$ near the point $t_0$. We partially follow this convention. For instance, in the first result, Theorem \ref{th1}, conditions are
given in terms of the zeros of the polynomial $z^m(b(z)-w_0)$. In contrast, LRG for regular Laurent polynomial symbols, \cite[Theorem 2.1]{gokuvi},
is based on certain properties of zeros of the whole family of polynomials $z^m(b(z)-w)$, $w\in\rho(T_b)$.
In the second result, Theorem \ref{weakllrg}, we obtain bound  \eqref{llrg}, but in a smaller neighborhood $\oo'\subset\oo$, and for all $w_0\in \pt\s(T_b)$ except for
a finite set of points on the boundary of the spectrum. The result is illustrated  by the example of the flower with three petals 
in the final section of the paper, see Figure \ref{fig1}, where LLRG is studied at the neighborhood of the origin (the intersection point for $b(\bt)$).

\section{Local regularity and LLRG}\label{s1}

We follow notation of \cite{gokuvi} for the algebraic equation
$$ P(z,w):=z^m(b(z)-w)=0, \qquad P(z,w)=b_k\,\prod_{j=1}^{m+k} (z-z_j(w)), $$
its divisors $Z(w)=\{z_j(w)\}_{j=1}^{m+k}$, and their partitions
\begin{equation*}
\begin{split} 
Z(w_0) &=Z_{in}(w_0)\cup Z_{un}(w_0)\cup Z_{ext}(w_0), \quad Z(w)=Z_{in}(w)\cup Z_{ext}(w), \quad w\notin b(\bt), \\
Z_{in}(v) &= Z(v)\cap\bd, \quad Z_{un}(v)=Z(v)\cap\bt, \quad Z_{ext}(v)=Z(v)\cap\bd_-,
\end{split}
\end{equation*}
for the unperturbed and perturbed divisors, respectively. Note that $Z_{un}(w_0)\not=~\emptyset$. 

The Wiener--Hopf factorization now looks as
\begin{equation}\label{wiho}
\begin{split}
b(z)-w &=b_k\,\prod_{z_i\in Z_{ext}}(z-z_i(w))\cdot \prod_{z_j\in Z_{in}}\lp 1-\frac{z_j(w)}{z}\rp\, \\
&=b_{ext}(z,w)b_{in}(z,w),
\end{split}
\end{equation}
where
\begin{equation}\label{inext}
b_{ext}(z,w)=b_k\,\prod_{z_i\in Z_{ext}}(z-z_i(w)), \quad b_{in}(z,w)=\prod_{z_j\in Z_{in}}\lp 1-\frac{z_j(w)}{z}\rp.
\end{equation}

The concept of regularity for Laurent polynomials, introduced in \cite[Definition 1.5]{gokuvi}, has its obvious local analogue: 
a symbol $b$ is locally regular at the point $w_0\in\pt\s(T_b)$ if at least one of $Z_{in}(w)$, $Z_{ext}(w)$ is separated from the unit circle 
$\bt$ as long as $w\in\oo$ (but not necessarily for all $w\in\rho(T_b)$).

As we will see shortly, the local regularity at $w_0$ implies LLRG($w_0$), in exactly the same way as the regularity implies LRG \cite[Theorem 2.1]{gokuvi}.
As usual, $|A|$ is the number of points in $A$, counting multiplicity.

\begin{theorem}\label{th1}
Given a Laurent polynomial $b$, and $w_0=b(t_0)\in\pt\s(T_b)$, let at least one of the following conditions holds
\begin{itemize}
\item[(i)]  $|Z_{un}(w_0)|=1$,
\item[(ii)] $|Z_{in}(w_0)|=m$,
\item[(iii)] $|Z_{ext}(w_0)|=k$.
\end{itemize}
Then $b$ is locally regular at $w_0$, and the LLRG($w_0$) holds. In particular, the latter is true whenever $w_0$ is neither a self-intersection point, 
nor a cusp $(b'(t_0)\not=0)$.
\end{theorem}

\begin{proof}
The idea is to watch over $Z_{un}(w_0)$ for small perturbations of $w_0$ within $\rho(T_b)$, taking into account continuity of divisors. Under
assumption (i), $Z_{un}(w_0)=\{z_n(w_0)\}$ is a single point, and it moves inside (outside) the unit disk $\bd$, so $Z_{ext}(w)$ ($Z_{in}(w)$) is
separated from $\bt$, and the local regularity follows.

Under assumption (ii), all $Z_{un}(w_0)$ moves outside $\bd$, since, by \cite[Lemma 1.4]{gokuvi}, $|Z_{in}(w)|=m$ for all $w\in\rho(T_b)$. Similarly, 
under assumption (iii), all $Z_{un}(w_0)$ moves inside $\bd$, since, by the same result, $|Z_{ext}(w)|=k$ for all $w\in\rho(T_b)$. 
So, under both (ii) and (iii), the local regularity holds.

To show that the local regularity at $w_0$ implies LLRG($w_0$), we proceed as in the proof of \cite[Theorem 2.1]{gokuvi}. 
The argument is based on the result of M.~Krein, see \cite[Theorem 1.15]{bosil}
\begin{equation*}
\bigl\|(T_b-w)^{-1}\bigr\|\le \bigl\|b_{in}^{-1}\bigr\|_\infty\,\cdot \bigl\|b_{ext}^{-1}\bigr\|_\infty,
\end{equation*}
$b_{in}$ and $b_{ext}$ are from \eqref{wiho}.

Assume first that $Z_{in}(w)$ is separated from $\bt$, \emph{i.e., }
$$ 1-|z_j(w)|>C_1>0, \qquad z_j\in Z_{in}(w), \quad w\in\oo. $$
In what follows, $C_j$ stand for some positive constants, which depend on $b$, $w_0$, $\ep$. Then
\begin{equation*}
\begin{split}
|b_{in}(t,w)| &=\prod_{z_j\in Z_{in}}|t-z_j(w)|\ge\prod_{z_j\in Z_{in}}(1-|z_j(w)|)\ge C_2, \\
\bigl\|b_{in}^{-1}\bigr\|_\infty &\le C_2^{-1}, \qquad w\in\oo.
\end{split}
\end{equation*}

The lower bound for $b_{ext}$ is simple
\begin{equation*}
\begin{split}
|b_{ext}(t,w)| &=\frac{|b(t)-w|}{\prod_{z_j\in Z_{in}}|t-z_j(w)|}\ge\frac{|b(t)-w|}{\prod_{z_j\in Z_{in}}(1+|z_j(w)|)} \\
&\ge\frac{|b(t)-w|}{2^m}\,,
\end{split}
\end{equation*}
and so
\begin{equation*} 
\bigl\|b_{ext}^{-1}\bigr\|_\infty\le \frac{C_3}{\|b(t)-w\|_\infty}\le \frac{C_3}{\di(w,\s(T_b))}\,, \quad w\in\oo, 
\end{equation*}
which implies LLRG($w_0$).

Secondly, assume that $Z_{ext}(w)$ is separated from $\bt$, \emph{i.e., } for $z_i\in Z_{ext}(w)$
$$ |z_i(w)|-1\ge C_3>0, \qquad w\in\oo. $$
Then
$$ |b_{ext}(t,w)|=|b_k|\,\prod_{z_i\in Z_{ext}}|t-z_i(w)|\ge |b_k|\,\prod_{z_i\in Z_{ext}}\bigl(|z_i(w)|-1\bigr)\ge C_4>0, \quad w\in\oo. $$
To obtain the lower bound for $b_{in}$, note that for $w\in\oo$, all the roots $z_n(w)$ are uniformly bounded, $\max_n|z_n(w)|\le C_5$, and so
$$ |b_{in}(t,w)|=\frac{|b(t)-w|}{|b_{ext}(t,w)|}\ge C_6|b(t)-w|, \qquad \|b_{in}^{-1}\|_\infty\le \frac{C_7}{\di(w,b(\bt))}\,, $$
and LLRG($w_0$) follows once again.  

The last statement is equivalent to (i). The proof is complete.
\end{proof}

\begin{remark}
The condition (i) means that the polynomial
$$ Q(z,w_0):=\frac{P(z,w_0)-P(t_0,w_0)}{z-t_0}=\frac{P(z,w_0)}{z-t_0}\not=0, \qquad z=t\in\bt. $$
If
$$ P(z,w_0)=\sum_{j=0}^{m+k}p_j(w_0)z^j, \qquad Q(z,w_0)=\sum_{j=0}^{m+k-1}q_j(w_0)z^j, $$
it is easy to see that
$$ q_j(w_0)=t_0^{-j-1}\,\sum_{i=j+1}^{m+k}p_i(w_0). $$
A simple sufficient condition for $Q\not=0$ on $\bt$ is
$$ |q_n|>\sum_{j\not=n}|q_j| $$
for some $n=0,1,\ldots,m+k-1$.
\end{remark}

\section{Weak LLRG}\label{s2}

The further partition of the divisors $Z(w)$, $w\not=w_0$, is important in this section
\begin{equation*}
Z(w)=\wih Z(w)\cup \wit Z(w). 
\end{equation*}
Here $\wih Z$ is the part of $Z$, which is {\sl close to the unit circle} $\bt$
\begin{equation}\label{cloroot}
\wih Z(w):=\{z_j(w)\in Z(w): 1-|z_j(w)|=o(1), \ w\to w_0\}, \ \ w\in B(w_0,\ep),
\end{equation}
and $\wit Z$ the part of $Z$, which {\sl stays away from} $\bt$
\begin{equation}\label{farroot}
\wit Z(w):=\{z_i(w)\in Z(w): \bigl|1-|z_j(w)|\bigr|\ge C_8>0\}, \ \ w\in B(w_0,\ep),
\end{equation}
(note that $z_n(w)\to z_n(w_0)$, $w\to w_0$). In other words, $\wih Z(w)$ arises from $Z_{un}(w_0)$, and $\wit Z(w)$ from $Z_{in}(w_0)\cup Z_{ext}(w_0)$.
So, it seems reasonable putting
$$ \wih Z(w_0):=Z_{un}(w_0), \qquad \wit Z(w_0):=Z_{in}(w_0)\cup Z_{ext}(w_0). $$

Denote also
\begin{equation*}
\wih Z_{in}(w):=Z_{in}(w)\cap \wih Z(w), \qquad \wit Z_{in}(w):=Z_{in}(w)\cap \wit Z(w), 
\end{equation*}
so $Z_{in}(w)=\wih Z_{in}(w)\cup \wit Z_{in}(w)$,
and, similarly, for $\wih Z_{ext}$, $\wit Z_{ext}$. Accordingly, the Wiener--Hopf factorization \eqref{wiho} is
\begin{equation}\label{wiho1}
\begin{split}
b(z)-w &=b_{ext}(z,w)\,b_{in}(z,w)=b_{ext}(z,w)\,\wih b_{in}(z,w)\,\wit b_{in}(z,w), \\
\wih b_{in}(z,w) &=\prod_{z_j\in \wih Z_{in}}\lp 1-\frac{z_j(w)}{z}\rp, \quad \wit b_{in}(z,w)=\prod_{z_j\in \wit Z_{in}}\lp 1-\frac{z_j(w)}{z}\rp.
\end{split}
\end{equation}

The divisors $Z(w)$ are called {\sl simple}, if all $z_j(w)$ are pairwise distinct (\emph{i.e., }no multiple roots). The complementary set
\begin{equation}\label{expset}  
K:=\{w=b(\l): \ b'(\l)=0\} 
\end{equation}
is obviously finite, cf. \cite[Lemma 11.4]{bogr}. 
Saying that $w_0\notin K$ ($Z(w_0)$ is simple) is equivalent to
\begin{equation}\label{nocus}
b'(z_j(w_0))\not=0, \qquad z_j\in Z(w_0).
\end{equation}
Certainly, the more natural assumption would be $b'(z_j(w_0))\not=0$, $z_j~\in Z_{un}(w_0)$, but, in some way, the roots inside 
$\bd$ play a role. By continuity, $Z(w)$ are simple for $w\in B(w_0,\ep)$ with small enough $\ep>0$, and 
\begin{equation}\label{separ}
0<C_9\le |z_i(w)-z_j(w)|\le C_{10}; \ i\not=j, \ z_i, z_j\in Z(w), \ w\in B(w_0,\ep).
\end{equation}

The roots $z_j\in Z(w)$ are analytic functions on $B(w_0,\ep)$, and, by the Inverse Function Theorem, 
\begin{equation}\label{invfun}
z_j(w)=z_j(w_0)+\frac{w-w_0}{b'(z_j(w_0))}+O\bigl((w-w_0)^2\bigr), \quad w\in B(w_0,\ep).
\end{equation}
So
\begin{equation}\label{controot}
0<C_{11}|w-w_0|\le |z_j(w)-z_j(w_0)|\le C_{12}|w-w_0|, \qquad w\in B(w_0,\ep).
\end{equation}
It follows from \eqref{invfun} that for $z_j\in \hat Z(w)$ and $w\in B(w_0,\ep)$
\begin{equation}\label{uncir}
|z_j(w)|^2=1+2\re\left\{\frac{\ovl{z_j(w_0)}}{b'(z_j(w_0))}\,(w-w_0)\right\}+O(|w-w_0|^2).
\end{equation}

For $w_0\in\pt\s(T_b)$, $z_j(w_0)\in Z_{un}(w_0)$, consider the domains (unless they are empty)
\begin{equation}\label{nontdom}
G_j(w_0):=\left\{w\in\bc: \left|\re\left\{\frac{\ovl{z_j(w_0)}}{b'(z_j(w_0))}\,(w-w_0)\right\}\right|>C_{13}|w-w_0|\right\}.
\end{equation}
By \eqref{controot} and \eqref{uncir}, the roots $z_j(w)\in \hat Z(w)$ lie in certain Stolz angles
(interior and exterior) for $w\in G_j(w_0)$, 
\begin{equation}\label{stolz}
\bigl|1-|z_j(w)|\bigr|>C_{14}|z_j(w)-z_j(w_0)|, \qquad z_j(w)\in \hat Z(w).
\end{equation}

Define the domain
\begin{equation}\label{optdom}
\oo'(w_0,\ep):=\bigcap_{z_j\in\hat Z(w)}\left (G_j(w_0)\bigcap\oo(w_0,\ep)\right )\subset\oo(w_0,\ep).
\end{equation}
The main result concerns the ``weak'' LLRG, that is, the local linear resolvent growth on the domain $\oo'(w_0,\ep)$.

\begin{theorem}\label{weakllrg}
Let $w_0\notin K$. Then for small enough $\ep>0$,
\begin{equation}\label{wllrg}
\bigl\|(T_b-w)^{-1}\bigr\|\le\frac{C(b,w_0,\ep)}{\di(w,\s(T_b))}\,, \qquad w\in\oo'(w_0,\ep).
\end{equation}
\end{theorem}
\begin{proof}
We begin with the explicit formula for the resolvent $(T_b-w)^{-1}$, see, \emph{e.g., } \cite[Theorem 3.3.6]{nik}, 
in terms of the Wiener--Hopf factorization \eqref{wiho}
\begin{equation}\label{expres}
(T_b-w)^{-1}h=\frac1{b_{ext}(t,w)}\,\bp_+\lp\frac{h(\cdot)}{b_{in}(\cdot,w)}\rp, \qquad h\in H^2,
\end{equation}
$\bp_+$ is the orthogonal projection of $L^2$ onto $H^2$. 

Since $K$ \eqref{expset} is finite, and $w_0\notin K$, we have
$$ B(w_0,\ep)\cap K=\emptyset $$
for small enough $\ep>0$, and \eqref{separ} holds.

The fraction in \eqref{expres} is
\begin{equation*}
\begin{split}
\frac{h(t)}{b_{in}(t,w)} &=\frac{t^m h(t)}{\prod_{z_j\in Z_{in}}(t-z_j(w))}=\sum_{z_j\in Z_{in}}a_j(w)\frac{t^m h(t)}{t-z_j(w)}\,, \\
a_j(w) &=\prod_{p\not=j}(z_j(w)-z_p(w))^{-1}, \quad z_j(w), z_p(w)\in Z_{in}.
\end{split}
\end{equation*}
Note that, due to \eqref{separ}, $|a_j(w)|\le C_{15}$ 
%\footnote{(for Stas) Here the simplicity of $Z(w_0)$ (and not just $Z_{un}(w_0)$) is needed.}
for all $w\in B(w_0,\ep)$.

The formula below is well known and can be checked directly,
$$ \bp_+\lp\frac{f(\cdot)}{t-\a}\rp=\frac{f(t)-f(\a)}{t-\a}\,, \qquad f\in H^2, \quad \a\in\bd. $$
Hence,
\begin{equation*}
\bp_+\lp\frac{h(\cdot)}{b_{in}(\cdot,w)}\rp=\sum_{z_j\in Z_{in}}a_j(w)\frac{t^mh(t)-z_j^m(w)h(z_j(w))}{t-z_j(w)}\,,
\end{equation*}
and \eqref{expres} looks as
\begin{equation*}
\begin{split}
(T_b-w)^{-1}h &=\frac1{b(t)-w}\sum_{z_j\in Z_{in}}a_j(w)\frac{b_{in}(t,w)}{t-z_j(w)}\bigl(t^mh(t)-z_j^m(w)h(z_j(w)\bigr)=I_1-I_2, \\
I_1 &=\frac1{b(t)-w}\sum_{z_j\in Z_{in}}a_j(w)\frac{b_{in}(t,w)}{t-z_j(w)}\,t^mh(t), \\
I_2 &=\frac1{b(t)-w}\sum_{z_j\in Z_{in}}a_j(w)\frac{b_{in}(t,w)}{t-z_j(w)}\,z^m_j(w)h(z_j(w)).
\end{split}
\end{equation*}
So,
\begin{equation}\label{normres} 
\bigl\|(T_b-w)^{-1}h\bigr\|_2\le \|I_1\|_2+\|I_2\|_2. 
\end{equation}
The bound for the first term on the RHS is simple. Indeed, as $|z_j(w)|\le1$, $z_j(w)\in Z_{in}$, we see that
$$ \left|\frac{b_{in}(t,w)}{t-z_j(w)}\right|\le 2^{m-1}, $$
and
\begin{equation}\label{ione}
\|I_1\|\le\frac{C_{16}\|h\|_2}{\min_{t\in\bt}|b(t)-w|}\le \frac{C_{16}}{\di(w,\s(T_b))}\|h\|_2\,.
\end{equation}

The calculation for the second term is more complex. We have
$$ \|I_2\|\le C_{17}\left\|\frac1{b(t)-w}\right\|_2\,\sum_{z_j\in Z_{in}}|h(z_j(w))|, $$
and, as is well known,
\begin{equation}\label{repr}
|h(\l)|\le\frac{\|h\|_2}{\sqrt{1-|\l|^2}}\,, \qquad h\in H^2, \quad \l\in\bd. 
\end{equation}
Recall that $Z_{in}=\tilde Z_{in}\cup \hat Z_{in}$. For the roots $z_j\in\tilde Z_{in}$, which are far from $\bt$, we simply have by \eqref{repr},
\begin{equation}\label{repr1} 
|h(z_j(w))|\le C_{18}\|h\|_2, \qquad z_j\in\tilde Z_{in}(w). 
\end{equation}

Let now $z_j\in\hat Z_{in}$. Here the Stolz angles \eqref{stolz} show up. In view of \eqref{stolz} and \eqref{controot},
\begin{equation}\label{repr2}
|h(z_j(w))|\le \frac{C_{19}\|h\|_2}{\sqrt{|w-w_0|}}\,, \qquad z_j\in\hat Z_{in}(w). 
\end{equation}

It remains only to obtain an appropriate bound for $\|(b(\cdot)-w)^{-1}\|_2$. Recall, that $Z(w)=\tilde Z(w)\cup\hat Z(w)$, and so
\begin{equation*}
\begin{split}
&{}\left\|\frac1{b(t)-w}\right\|_2^2 =\frac1{|b_k|}\int_{\bt}\prod_{j=1}^{m+k}\frac1{|t-z_j(w)|^2}\,m(dt)
\le C_{20}\int_{\bt}\prod_{z_i\in \hat Z}\frac1{|t-z_i(w)|^2}\,m(dt) \\
&=C_{20}\left\|\sum_{z_i\in\hat Z}\frac{d_i(w)}{t-z_i(w)}\right\|_2^2, \ \ d_i(w)=\prod_{p\not=i}(z_i(w)-z_p(w))^{-1}, \ \ z_j, z_p\in\hat Z. 
\end{split}
\end{equation*}
Hence, due to \eqref{separ}, \eqref{stolz} and \eqref{controot},
$$ \left\|\frac1{b(t)-w}\right\|_2\le C_{21}\sum_{z_i\in\hat Z}\frac1{\sqrt{|1-|z_j(w)|^2|}}\le\frac{C_{22}}{\sqrt{|w-w_0|}}. $$
Finally, by \eqref{repr2},
\begin{equation}\label{itwo}
\|I_2\|\le\frac{C_{23}}{|w-w_0|}\|h\|_2\le\frac{C_{23}}{\di(w,\s(T_b))}\|h\|_2.
\end{equation}
The combination of \eqref{normres}, \eqref{ione} and \eqref{itwo} completes the proof.
\end{proof}

\begin{remark}
We can weaken a bit the assumption $w_0\notin K$ to the following one: the part $Z_{un}(w_0)\cup Z_{in}(w_0)$ is simple (we allow multiple roots
outside the unit disk).
\end{remark}

\section{Example}\label{s2}

We consider a simple example 
%\footnote{(for Stas). If we keep this example, the figure is needed}
of the symbol with self-intersection
\begin{equation}\label{exam1}
b(z):=z^{-1}+z^2, \qquad w_0=0.
\end{equation}
The image $b(\bt)=\s_{ess}(T_b)=\pt(\s(T_b))$ is a flower with three petals, and the union of these petals constitutes the spectrum $\s(T_b)$, see Figure \ref{f1}.

The resolvent set $\rho(T_b)$ contains three angles with the vertex at the origin
\begin{equation*}
\begin{split} 
\gg_1&:=\Bigl\{\frac{\pi}6<\p<\frac{\pi}2\Bigr\}, \qquad \gg_2:=\Bigl\{\frac{5\pi}6<\p<\frac{7\pi}6\Bigr\}, \\ 
\gg_3&:=\Bigl\{\frac{3\pi}2<\p<\frac{11\pi}6\Bigr\}, \qquad \gg:=\bigcup_{j=1}^3\gg_j, \quad w=|w|e^{i\p}.
\end{split}
\end{equation*}
For an arbitrary (small enough) $\d>0$, consider the interior angles
\begin{equation*} 
\begin{split} \gg_1^\d &:=\Bigl\{\frac{\pi}6+\d<\p<\frac{\pi}2-\d\Bigr\}, \qquad \gg_2^\d:=\Bigl\{\frac{5\pi}6+\d<\p<\frac{7\pi}6-\d\Bigr\}, \\ 
\gg_3^\d&:=\Bigl\{\frac{3\pi}2+\d<\p<\frac{11\pi}6-\d\Bigr\}, \qquad \gg^\d:=\bigcup_{j=1}^3\gg_j^\d,
\end{split}
\end{equation*}

\begin{figure}[t]
\centering
 \includegraphics[width=7cm]{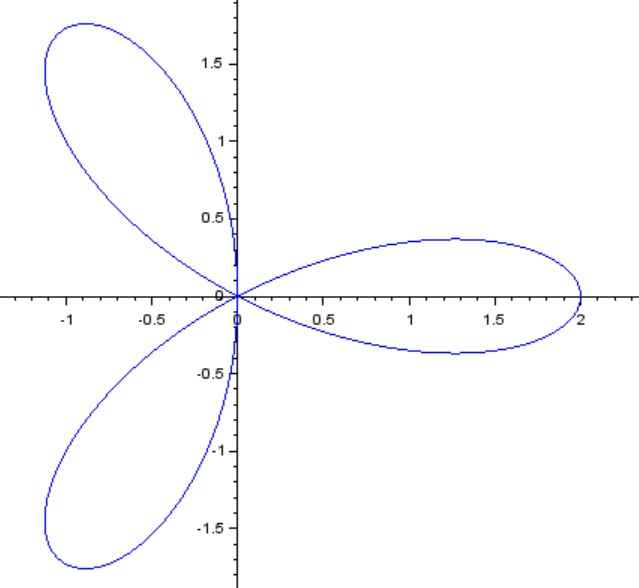}
 \caption{The curve $b(\bt)$ for $b(z)=z^{-1}+z^2, \, z\in \bt$. 
 }\label{f1}
  
\end{figure}

The unperturbed divisor is simple
$$ Z(0)=\{z_j(0)=e^{i t_j}\}_{j=1}^3: \qquad t_1=-\frac{\pi}3, \quad t_2=\pi, \quad t_3=\frac{\pi}3. $$
It is easy to see that $b'(z_j(0))=3z_j(0)$, $j=1,2,3$, and so
$$ \re\left\{\frac{\ovl{z_j(0)}}{b'(z_j(0))}\,w\right\}=\frac{|w|}3\cos(\p-2t_j). $$
An elementary calculation shows that for each fixed and small enough $\d>0$
$$ |\cos(\p-2t_j)|\ge C(\d)>0, \qquad w=|w|e^{i\p}\in\gg^\d, \quad j=1,2,3, $$
(this is, in general, not the case for $\d=0$). Hence, by \eqref{nontdom} and \eqref{stolz},
$$ \gg^\d\cap B(0,\ep)\subset\oo'(0,\ep), \qquad \forall\d>0. $$
Theorem \ref{wllrg} then claims that for small enough $\ep>0$,
\begin{equation}\label{exam2}
\bigl\|(T_b-w)^{-1}\bigr\|\le\frac{C(b,\d,\ep)}{\di(w,\s(T_b))}, \qquad w\in\gg^\d\cap B(0,\ep), \quad \forall\d>0.
\end{equation}
%{\bf Problem}. Is it true that \eqref{exam2} holds with $\d=0$?


\begin{thebibliography}{99}

\bibitem{bogr}
A. B\"ottcher and S. Grudsky, {\it Spectral Properties of Banded Toeplitz Matrices}, SIAM, Philadelphia 2005.

\bibitem{bosil}
A. B\"ottcher and B. Silbermann, {\it Introduction to Large Truncated Toeplitz Matrices}, Springer Science NY, 1999.

\bibitem{gokuvi}
L. Golinskii, S. Kupin and A. Vishnyakova, On the growth of resolvent of Toeplitz operators, MPAG, 2024.

\bibitem{ha}
P. Halmos, {\it A Hilbert Space Problem Book}, Springer--Verlag, Berlin, 1982.

\bibitem{nik}
N. Nikolski, {\it Toeplitz Matrices and Operators}, Cambridge Studies in Advanced Mathematics, v.182, 2020.
\end{thebibliography}
\end{document}